\documentclass[12pt, a4paper]{article}
\usepackage[T1]{fontenc}
\usepackage[cp1250]{inputenc}
\usepackage{amssymb}
\usepackage{amsmath}
\usepackage{amsfonts}
\usepackage{amsthm}
\usepackage{graphicx}
\usepackage[all]{xy}
\usepackage{authblk}

\newcommand{\set}[1]{\left\{#1\right\}}
\newcommand{\abs}[1]{\left|#1\right|}
\newcommand{\Nat}{\mathbb{N}}

\newcommand{\Int}{\mathbb{Z}}

\newcommand{\CB}{\mathcal{B}}
\newcommand{\CC}{\mathcal{C}}

\newcommand{\CM}{\mathcal{M}}
\newcommand{\CT}{\mathcal{T}}

\newcommand{\CS}{\mathcal{S}}

\newcommand{\Bor}{\mathit{Bor}}

\newcommand{\fr}{\mathsf{fr}}

\newcommand{\Fo}{F{\o}lner }

\newcommand{\sd}{\bigtriangleup}
\newcommand{\eps}{\varepsilon}
\theoremstyle{plain}
\newtheorem{thm}{Theorem}[section]

\newtheorem{cor}[thm]{Corollary}

\newtheorem{lem}[thm]{Lemma}
\newtheorem{prop}[thm]{Proposition}
\newtheorem{defn}[thm]{Definition}
\newtheorem{rem}[thm]{Remark}
\newtheorem{fct}[thm]{Fact}
\numberwithin{equation}{section}

\renewcommand{\thefootnote}{}

\title{Faces of simplices of invariant measures for actions of amenable groups}
\author{Bartosz Frej, Dawid Huczek}
\affil{\textit{Faculty of Pure and Applied Mathematics,
Wroc{\l}aw University of Science and Technology, 
Wybrze\.{z}e Wyspia\'{n}skiego 27,
50-370 Wroc{\l}aw, Poland}\\
\texttt{Bartosz.Frej@pwr.edu.pl, Dawid.Huczek@pwr.edu.pl}}

\date{}
\begin{document}

\maketitle 

\renewcommand{\thefootnote}{}
\footnote{2010 \emph{Mathematics Subject Classification}: Primary 37A35; Secondary 37B40.}
\footnote{\emph{Key words and phrases}: invariant measure, periodic measure, block, Choquet simplex, amenable group, group action, symbolic system, Cantor space.}
\footnote{Research of both authors is supported from resources for science in years 2013-2018
as research project (NCN grant 2013/08/A/ST1/00275, Poland).}

\renewcommand{\thefootnote}{\arabic{footnote}}
\setcounter{footnote}{0}

\abstract{We extend the result of \cite{D2} to the case of amenable group actions, by showing that every face in the simplex of invariant measures on a zero-dimensional dynamical system with free action of an amenable group $G$ can be modeled as the entire simplex of invariant measures on some other zero-dimensional dynamical system with free action of $G$. This is a continuation of our investigations from \cite{FH}, inspired by an earlier paper \cite{D1}.}

\section{Introduction} 
Let $X$ be a Cantor space i.e., a compact, metrizable, zero-dimensional perfect space, and let $G$ be a countable amenable group acting on $X$ via homeomorphisms $\varphi_g, g\in G$. Amenability of $G$ means that there exists a sequence of finite sets $F_n\subset G$ (called a \emph{F\o lner sequence}, or the sequence of F\o lner sets), such that for any $g\in G$ we have
\[\lim_{n\to\infty}\frac{\abs{gF_n\sd F_n}}{\abs{F_n}}= 0,\]
where $gF=\set{gf:f\in F}$, $\abs{\cdot}$ denotes the cardinality of a set, and $\sd$ is the symmetric difference. 
The action of $G$ is \emph{free} if the equality $gx=x$ for any $g\in G$ and $x\in X$ implies that $g$ is the neutral element of $G$.
It is well known that one can represent the system $(X,G)$ as an inverse limit $\lim\limits_{\leftarrow} X_j\subset \prod_{j\in\Nat}X_j$ where each $X_j$ is a group subshift on finitely many symbols i.e. a subset of some ${\Lambda_j}^G$, $|\Lambda_j|<\infty$, with the action defined by $gx(h)=x(hg)$. 
We will often refer to this inverse limit as a so called array system --- an element of $X$ in this interpretation is a map $x(\cdot,\cdot)$ on $G\times \Nat$, where $x(\cdot,j)\in X_j$. We will call such a map an array and from now on we will assume that our system is in array representation. By an $(F,k)$-block  we mean a map $B\colon F\times[1,k]\to\bigcup_j\Lambda_j$, where $F$ is a finite subset of $G$ (which will occasionally be called the \emph{shape} of a block), $k$ is a positive integer and $[1,k]$ is an abbreviation for $\{1,...,k\}$. If $E$ is a subset of the domain of a block $B$ then by $B[E]$ we will denote a restriction of $B$ to $E$. By abuse of the notation, we will mean by $|B|$ the cardinality of the shape of $B$. We will  use the same letter to denote both a block and a cylinder set induced by this block---the exact meaning is always clear from the context. A block $B$ \emph{occurs in $X$} if $B$ is a restriction of some $x\in X$.

Let $K$ be an abstract metrizable Choquet simplex.
\begin{defn}

\begin{enumerate}
\item An \emph{assignment} on $K$ is a function $\Phi$ defined on $K$ such that 
for each $p\in K$, the value of $\Phi(p)$ is a measure-preserving group action $(X_p,\Sigma_p,\mu_p,G_p)$, where $(X_p,\Sigma_p,\mu_p)$ is a standard
probability space. 

\item Two assignments $\Phi$ on $K$ and $\Phi'$ on $K'$  are \emph{equivalent} if there exists an affine
homeomorphism $\pi: K \to K'$ such that $\Phi(p)$ and $\Phi'(\pi(p))$ are isomorphic for every $p \in K$.

\item If $(X, G)$ is a continuous group action on a compact metric space $X$ then the set of all $G$-invariant measures supported
by $X$, endowed with the weak* topology of measures, is a Choquet simplex, and the
assignment by identity $\Phi(\mu) = (X,\Bor_X, \mu, G)$ (where $\Bor_X$ is the Borel sigma-field) is
 \emph{the natural assignment} of $(X,G)$.
\end{enumerate}
\end{defn}

In the current article we aim to prove the following:
\begin{thm}\label{thm:main}
Let $X$ be a Cantor system with free action of an amenable group $G$ and let $K$ be a face in the simplex $\CM_G(X)$ of $G$-invariant measures of $X$. There exists a Cantor system $Y$ with free action of $G$, such that the natural assignment on $Y$ is equivalent to the identity assignment on $K$.
\end{thm}

In case of actions of $\Int$ the theorem was proved in \cite{D2} and the key tool used there was approximation of an arbitrary ergodic measure by a block (periodic) measure. 
Density of periodic measures in the set of all invariant measures is usually a desired property and was proved to be true in various cases, e.g. for systems with specification property (see \cite{DGS}). In case of a one-dimensional subshift one can construct a periodic measure by choosing a block $B$ occurring in a system and uniformly distributing a probability mass on the orbit of a sequence obtained by periodic repetitions of $B$. Such a sequence need not be an element of a subshift (and the measure need not belong to its simplex of invariant measures), still it may give a useful approximation of a measure under consideration. For actions of groups other than $\Int$ (even $\Int^d$) this procedure usually cannot be performed, roughly saying, because of irregular shapes of blocks, and the notion of a block measure seems to be obscure. We devote the next section to implementing it in our setup, but before we proceed, we recall a few facts about F\o lner sequences.

In any amenable group there exists a F\o lner sequence with the following additional properties (see \cite{E}):
\begin{enumerate}
\item $F_n\subset F_{n+1}$ for all $n$, 
\item $e\in F_n$ for all $n$ ($e$ denotes the neutral element of $G$),
\item $\bigcup_{n\in \Nat} F_n=G$,
\item $F_n=F_n^{-1}$ for all $n$.
\end{enumerate}
A F\o lner sequence $F_n$ is \emph{tempered} if for some $C>0$ and all $n$,
\[
\abs{\bigcup_{k\leq n} F_k^{-1}F_{n+1}} \leq C|F_{n+1}|.
\]
\begin{prop}[\cite{L}]
Every F\o lner sequence $F_n$ has a tempered subsequence.
\end{prop}
Throughout this paper, we will assume that the F\o lner sequence which we use is tempered and has all the above properties. We recall the pointwise ergodic theorem for amenable groups.

\begin{thm}[\cite{L}] \label{ergodic_thm}
Let $G$ be an amenable group acting
ergodically on a measure space $(X, \mu)$, and let $F_n$ be a tempered F\o lner sequence.
Then for any $f \in L^1(\mu)$,
\[
\lim_{n\to\infty} \frac1{|F_n|}\sum_{g\in F_n} f(gx) = \int f \,d\mu\qquad a.e.
\]
\end{thm}

If $F$ and $A$ are finite subsets of $G$ and $0<\delta<1$, we say that $F$ is \emph{$(A,\delta)$-invariant} if
\[\frac{\abs{F\sd AF}}{\abs{F}}< \delta,\]
where $AF=\set{af:a \in A,f\in F}$. Observe that if $A$ contains the neutral element of $G$, then $(A,\delta)$-invariance is equivalent to the simpler condition
\[\abs{AF}< (1+\delta)\abs{F}.\]
It is not hard to observe that if $F$ is $(A,\delta)$-invariant then
\[
|\{f\in F: Af\cap F^c\not=\emptyset\}| < \delta|A||F|.
\]

If $(F_n)$ is a F\o lner sequence, then for every finite $A \subset G$ and 
every $\delta>0$ there exists an $N$ such that for $n>N$ the sets $F_n$ are 
$(A,\delta)$-invariant.

\begin{defn}
	For $S\subset G$ and a finite, nonempty $F\subset G$ denote
	\[
	\underline D_F(S)=\inf_{g\in G} \frac{|S\cap Fg|}{|F|}, \ \ \ \overline 
	D_F(S)=\sup_{g\in G} \frac{|S\cap Fg|}{|F|}.
	\]
	If $(F_n)$ is a F\o lner sequence then we define two values
	\[
	\underline D(S)=\limsup_{n\to\infty} \underline D_{F_n}(S) \ \ \ \text{ and 
	} \ \ \ \overline D(S)=\liminf_{n\to\infty} \overline D_{F_n}(S),
	\]
	which we call the \emph{lower} and \emph{upper} \emph{Banach densities} of 
	$S$, respectively.
\end{defn}
Note that $\overline D(S) = 1- \underline D(G\setminus S)$. We recall the 
following standard fact:

\begin{fct}\label{bd}
	Regardless of the set $S$, the values of $\underline D(S)$ and $\overline 
	D(S)$ do not depend on the
	F\o lner sequence, the limits superior and inferior in the definition are 
	in fact limits, and moreover
	\begin{align*}
		\underline D(S) &= \sup\{\underline D_F(S): F\subset G, F \text{ is 
		finite}\}\ \ \ \text{ and } \\
		\overline D(S) &= \,\inf\,\{\overline D_F(S): F\subset G, F \text{ is 
		finite}\}\ge \underline D(S).
	\end{align*}
\end{fct}

% ======================================================================

\section{Block measures} 

We will explicitly define a metric consistent with the weak* topology on the set of probability measures on $X$, represented as an array system. First, let $\CB_k$ be the family of all blocks with domain $F_k\times[1,k]$, occurring in $X$, and let 
\[d_k(\mu,\nu)=\frac{1}{\abs{\CB_k}}\sum_{B\in\CB_k}\abs{\mu(B)-\nu(B)}.\]
Now let
\[d(\mu,\nu)=\sum_{k=1}^\infty\frac{1}{2^k}d_k(\mu,\nu).\]
Note that we may assume that $\CB_k$ consists only of blocks which yield cylinders of positive measure for some ergodic measure $\mu$.

For the sake of convenience, we introduce a notion of ``distance'' between a block and a measure.
Let $B$ be a block occurring in $X$, with domain $F\times[1,k]$ for some $F\subset G$ and $k\in \Nat$. For any block $C$ with domain $F_j\times[1,j]$, where $j\leq k$, we can define the frequency of $C$ in $B$ in the following way: let 
\begin{gather*}
N_F(F_j)=\abs{\set{g\in F:F_jg\subset F}}\\
N_B(C) = \abs{\set{g\in F: F_jg\subset F \text{ and }B[F_jg\times[1,j]]=C}}
\end{gather*}
and if $N_F(F_j)>0$ let 
\[
\fr_B(C)=\frac{N_B(C)}{N_F(F_j)}.
\]
Otherwise let $\fr_B(C)=0.$ 

We say that $A$ is a \emph{$(1-\delta)$-subset of $F$} if $A\subset F$ and $|A| \geq (1-\delta)|F|$. 
By a standard argument we can draw from the pointwise ergodic theorem \ref{ergodic_thm} the following corollary.
\begin{fct}	\label{pet_consequence}  
Let $\mu$ be an ergodic measure on $X$.
For every $\eps$ and $j$ we can find $n$ and $\eta$ such that if $F$ is a $(1-\eta)$-subset of $F_m$, $m\geq n$, then for some block $C$ with domain $F\times [1,j]$ we have $\abs{\fr_C(D)-\mu(D)}<\eps$ for every block $D$ with domain 
$F_i\times[1,i]$, $i=1,\ldots,j$.
\end{fct}

We can now define the distance between a block and a measure: let
\[
d_k(B,\nu)=\frac{1}{\abs{\CB_k}}\sum_{D\in\CB_k}\abs{\fr_B(D)-\nu(D)},
\]
and let
\[
d(B,\nu)=\sum_{k=1}^\infty\frac{1}{2^k}d_k(B,\nu).
\]

\begin{rem} \label{freq_then_dist}
Let $\eps$ be a positive number.
To ensure that $d(B,\nu)<\eps$ it is enough to verify that if $j$ satisfies $\sum_{k=j+1}^\infty\frac{1}{2^k}<\frac{\eps}{2}$ then for any block $D\in \CB_i$, where $i=1,...,j$,
\[
\abs{\fr_B(D)-\nu(D)}<\frac{\eps}{2j}.
\]\end{rem}
Indeed, in this case we have $d_k(B,\nu)<\frac{\eps}{2j}$ and
\[
d(B,\nu)=\sum_{k=1}^j\frac{1}{2^k} d_k(B,\nu)+\sum_{k=j+1}^\infty \frac{1}{2^k}d_k(B,\nu)\\
<\sum_{k=1}^j\frac{\eps}{2j}+\sum_{k=j+1}^\infty\frac{1}{2^k}<\eps.
\]

\begin{lem}\label{block_measure}
For any $\eps>0$ and any positive integer $j$ there exists $\delta$ such that if $F$ is an $(F_j,\delta)$-invariant set and $B$ is a block with domain $F\times[1,j]$, then there exists a probability measure $\mu_B$ such that 
\[
\abs{\fr_B(D)-\mu_B(D)}<\frac{\eps}{2j}
\]
for any block $D\in \CB_i$, where $i=1,...,j$.

Consequently, for any $\eps>0$ and sufficiently large $j$ there exists $\delta$ such that if $F$ is an $(F_j,\delta)$-invariant set and $B$ is a block with domain $F\times[1,j]$, then there exists a probability measure $\mu_B$ such that $d(B,\mu_B)<\eps$.
\end{lem}
\begin{proof}
Let  $\Delta_j=\Lambda_1\times...\times\Lambda_j$. The full shift $\Delta_j^G$ is a Cantor set on which we have the uniform Bernoulli probability measure $\lambda$ which assigns equal measures $M_j$ to all cylinders with domain $F_j\times[1,j]$. We shall define $\mu_B$ by specifying its density $f_B$ with respect to $\lambda$. $f_B$ will be constant on cylinders with domain $F_j\times[1,j]$: on each such cylinder associated with a block $C$ let $f_B(x)=\frac{1}{M_j}\fr_B(C)$. Obviously $d_j(\mu_B,B)=0$. 

We will now estimate $d_i(\mu_B,B)$ for $i<j$. Let $D$ be any block from $\CB_i$. Let $\CC_j$ be the family of (distinct) blocks from $\CB_j$ such that $D=\bigcup_{C\in\CC_j}C$. We have:
\[\mu_B(D)=\sum_{C\in\CC_j}\mu_B(C)=\sum_{C\in\CC_j}\fr_B(C),\]
and we need to show that the latter quantity is close to $\fr_B(D)$. 
Since $F$ is $(F_j,\delta)$-invariant, the set $\set{g:F_jg\subset F}$ is a 
$(1-\delta\abs{F_j})$-subset of $F$.
Consequently, the set $\set{g:F_ig\subset F}$ also is a 
$(1-\delta\abs{F_j})$-subset of $F$, being a superset of the former. 
For $C\in\CC_j$, let $F_C$ be the set of $g$ such that $F_jg\subset F$, $B[F_jg\times[1,j]]=C$ (and automatically $B[F_ig\times[1,i]]=D$). That way we can represent $\set{g:F_ig\subset F,B[F_ig\times[1,i]]=D}$ as the following disjoint sum:
\begin{multline*}
\set{g:F_ig\subset F,\ B[F_ig\times[1,i]]=D}=\\
=\bigcup_{C\in\CC_j}F_C \cup\set{g:F_ig\subset F,\ F_jg\cap 
F^c\not=\emptyset,\ B[F_ig\times[1,i]]=D}.
\end{multline*}
Taking cardinalities and dividing by $N_F(F_i)$, we obtain
\begin{multline*}
\fr_B(D)=\sum_{C\in\CC_j}\frac{N_F(F_j)}{N_F(F_i)}\fr_B(C)+\\
+\frac1{N_F(F_i)} \abs{\set{g:F_ig\subset F,\ F_jg\cap 
F^c\not=\emptyset,\ B[F_ig\times[1,i]]=D}}\\
\end{multline*}
Since both $F_j$ and $F_i$ are $(1-\delta\abs{F_j})$-subsets of $F$, we have  
$\frac{N_F(F_j)}{N_F(F_i)}\geq 1-\delta\abs{F_j}$ and 
\begin{multline*}
\frac1{N_F(F_i)} \abs{\set{g:F_ig\subset F,F_jg\cap 
F^c\not=\emptyset,\ B[F_ig\times[1,i]]=D}} \leq\\
\leq \frac1{N_F(F_i)} \abs{\set{g\in F:F_jg\cap F^c\not=\emptyset}} \\ 
\leq \frac{\delta\abs{F_j} |F|}{(1-\delta\abs{F_j})|F|} = 
\frac{\delta\abs{F_j}}{1-\delta\abs{F_j}}.
\end{multline*}
If $\delta$ is small enough then the expression can be arbitrarily close to $0$, while $\frac{N_F(F_j)}{N_F(F_i)}$ can be arbitrarily close to $1$, so we can assume that
\[\abs{\fr_B(D)-\mu_B(D)}=\abs{\fr_B(D)-\sum_{C\in\CC_j}\fr_B(C_j)}<\frac{\eps}{2j}.\]
The second assertion follows by remark \ref{freq_then_dist}.
\end{proof}
Note that in the above lemma $\delta$ may be as small as we want.

\begin{cor}	\label{approx_ergodic}
Let $X$ be a zero-dimensional dynamical system with the action of an amenable group $G$ and let $\mu$ be an ergodic measure on $X$. For any $\eps>0$ and any sufficiently large $j$ there exists $\delta>0$ and a block $C$ whose domain is a $(F_j,\delta)$-invariant set such that the measure $\mu_C$ (as defined in lemma \ref{block_measure}) satisfies $d(\mu_C,\mu)<\eps$.
\end{cor}
\begin{proof}
Choose $\delta$ and $j$ from lemma \ref{block_measure} with $\frac{\eps}{2}$ replacing $\eps$. Assume also that $\sum_{k=j+1}^\infty\frac{1}{2^k}<\frac{\eps}{4}$. By fact \ref{pet_consequence}, we can find a block $C$ on a $(F_j,\delta)$-invariant domain $F$, such that $\abs{\fr_C(D)-\mu(D)}<\frac{\eps}{4j}$ for any block $D$ with domain $F_i\times[1,i]$, $i=1,\ldots,j$. By remark \ref {freq_then_dist} we see that $d(C,\mu)<\frac{\eps}{2}$, and directly from lemma \ref{block_measure} also $d(C,\mu_C)<\frac{\eps}{2}$, therefore $d(\mu_C,\mu)<\eps$. 
\end{proof}
For actions of $\Int$, it is a well-known fact that if a sufficiently long block $C$ is a concatenation of shorter blocks $B_1,B_2,\ldots,B_n$ of equal length, then the probability measure $\mu_C$ (which for actions of $\Int$ can easily be assumed to be shift invariant) can be arbitrarily close to the arithmetic average of the measures $\mu_{B_i}$. An analogous claim can be made for the action of any amenable group $G$; however the lack of a natural way to decompose a subset of $G$ into smaller sets requires the use of quasitilings.

\begin{defn}
	A \emph{(static) quasitaling} of a group $G$ is a family $\CT$ of finite 
	subsets of 
	$G$ (called \emph{tiles}), for which there exist a family 
	$\CS(\CT)=\set{S_1,S_2,\ldots,S_n}$ of finite subsets of $G$ (called 
	\emph{shapes}) and a family $\CC(\CT)=\set{C_1,C_2,\ldots,C_n}$ of subsets 
	of $G$ (called \emph{centers}), such that every $T\in \CT$ has a unique 
	representation $T=S_ic$ for some $i\in\set{1,\ldots,n}$ and $c\in C_i$. 
\end{defn}
Note that every quasitiling can be seen as a symbolic element  
$\CT\in\set{0,1,\ldots,n}^G$, such that $\CT(g)=i$ if $g\in C_i$ for some $i$, 
and $\CT(g)=0$ otherwise.
\begin{defn}
A quasitiling $\CT$ is:
\begin{enumerate}
	\item \emph{disjoint}, if the tiles are pairwise disjoint;
	\item \emph{$\alpha$-covering}, if the union of all tiles has lower Banach 
	density at least $\alpha$.
	\item \emph{congruent} with a quasitiling $\CT'$, if for any two tiles $T\in 
	\CT,T'\in\CT'$ we have either $T\supset T'$ or $T\cap T'=\emptyset $.
\end{enumerate}
\end{defn}

Let $(X,G)$ be a topological dynamical system. Suppose we assign to every $x\in 
X$ a quasitiling $\CT(x)$ of $G$, with the same set of shapes $S_1,\ldots,S_n$ for all 
$x$. This induces a map $x\mapsto\CT(x)$ which can be seen as a map from 
$(X,G)$ 
into $\set{0,1,\ldots,n}^G$ with the shift action. If such a map is a factor 
map 
(i.e. if it is continuous and commutes with the dynamics), we call it a 
\emph{dynamical quasitiling}. A dynamical quastiling is said to be disjoint 
and/or $\alpha$-covering, if $\CT(x)$ has the respective property for every 
$x$.  

 Any $(T,k)$-block whose shape $T$ belongs to a quasitiling $\CT$ will be 
 called a \emph{$(\CT,k)$-block}.

\begin{lem}	\label{concatenation}
For any $\eps>0$ there exist $j\in\Nat$ and $\delta>0$ such that if $\CT$ is a disjoint quasitiling by $(F_j,\delta)$-invariant sets, and $C$ is a block with domain $H\times[1,j]$ such that some disjoint union of tiles $T_1,T_2,\ldots,T_n$ of $\CT$ is a $(1-\delta)$-subset of $H$, then the probability measure $\mu_C$ is $\eps$-close to the average of the measures associated with blocks over individual tiles, i.e. if we denote by $B_i$ the block with domain $T_i\times[1,j]$,
\[d\left(\mu_C,\frac{1}{\sum_{i=1}^n\abs{T_i}}\sum_{i=1}^n \abs{T_i}\mu_{B_i}
\right)<\eps.\]

\end{lem}
\begin{proof}
Applying lemma \ref{block_measure}, for any $j$ there is $\delta_j$ such that for any block $B$ with domain $F\times[1,j]$, where $F$ is a $(F_j,\delta_j)$-invariant set, and for any block $D$ with domain $F_i\times[1,i]$, $i\leq j$, we have $\abs{\fr_B(D)-\mu_B(D)}<\frac{\eps}{8j}$.
%, and such that if $C$ is any block with domain $F\times[1,j]$, where $F$ is an $(F_j,\delta)$-invariant set, then $d(C,\mu_C)<\frac{\eps}{2}$.
Let $H$ be a subset of $G$ and let $\CT$ be a quasitiling of $G$ by $(F_j,\delta)$-invariant sets for some $\delta > 0$. Suppose that the union $\bigcup_{i=1}^n T_i$ is a $(1-\delta)$-subset of $H$ for some pairwise disjoint tiles $T_1,T_2,\ldots,T_n$ belonging to $\CT$. For every $k\leq j$ let us define  the set
\[
E_k=\set{h\in H: \forall i\ F_kh\cap T_i^c \not=\emptyset}
\]
Then 
\begin{eqnarray*}
|E_k| & \leq & \sum_{i=1}^n |\set{h\in T_i: \forall i\ F_kh\cap T_i^c \not=\emptyset}| + |H\setminus \bigcup_{i=1}^n T_i|\\
& \leq & \sum_{i=1}^n \delta|F_k||T_i| + \delta|H| \leq \delta|H|(|F_k|+1),
\end{eqnarray*}
hence $N_H(F_k) \geq |H|-|E_k| \geq |H|(1-\delta(1+|F_k|))$.
Clearly, we can demand that $\delta<\delta_j$ (further restrictions will follow). Note  that since each $T\in\CT$ is $(F_j,\delta)$-invariant, for any block $B$ whose domain is a tile of $\CT$ the measure $\mu_B$ is well-defined.

Now, let $C$ be a block with domain $H\times[1,j]$, where $H$ is $(1-\delta)$-tiled by $T_1,...,T_n$, and let $C[T_i]=B_i$. For any $k\leq j$ and for any block $D$ with domain $F_k\times[1,k]$ we have:
\[
N_C(D)=\sum_{i=1}^nN_{B_i}(D)+N_{E_k}(D) %=\sum_{i=1}^n\fr_{B_i}(D)\abs{T_i}+N_{E_k}(D).
\]
%Dividing both sides of this equality by $\abs{H}$, we obtain, with use of triangle inequality,
Therefore, using the traingle inequality,
\begin{multline*}
\abs{\fr_C(D) - \frac{1}{\sum_{i=1}^n\abs{T_i}}\sum_{i=1}^n\fr_{B_i}(D)\abs{T_i}} =\\
 =\abs{\frac{\sum_{i=1}^nN_{B_i}(D)+N_{E_k}(D)}{N_H(F_k)} - \frac{1}{\sum_{i=1}^n\abs{T_i}}\sum_{i=1}^n\fr_{B_i}(D)\abs{T_i}} \\
\leq \frac{N_{E_k}(D)}{N_H(F_k)} + \sum_{i=1}^nN_{B_i}(D)\cdot\left|\frac{1}{N_H(F_k)} - \frac{1}{\sum_{i=1}^n\abs{T_i}}\right| +\\
+ \frac1{\sum_{i=1}^n\abs{T_i}}\cdot\sum_{i=1}^n \left(N_{B_i}(D)\left|1 - \frac{\abs{T_i}}{N_{T_i}(F_k)}\right|\right)
\end{multline*}
We can estimate that $\frac{N_{E_k}(D)}{N_H(F_k)} < \delta(|F_k|+1)$, $\left|\frac{1}{N_H(F_k)} - \frac{1}{\sum_{i=1}^n\abs{T_i}}\right| \leq \frac{1}{N_H(F_k)}\frac{\delta(|F_k|+2)}{1-\delta}$ and $\left|1 - \frac{\abs{T_i}}{N_{T_i}(F_k)}\right| \leq \frac{\delta|F_k|}{1-\delta|F_k|}$, so the whole expression can be made smaller than $\frac{\eps}{8j}$ by appropriate choice of a small $\delta$.

Now, for every $B_i$ we have $\fr_{B_i}(D)$ is approximately equal to $\mu_{B_i}(D)$ with error $\frac{\eps}{8j}$, and this approximation is preserved by the weighted average we have obtained, therefore
\[
\abs{\fr_C(D)-\frac{1}{\sum_{i=1}^n\abs{T_i}}\sum_{i=1}^n\abs{T_i}\mu_{B_i}(D)} < \frac{\eps}{4j}
\]
If $j$ is sufficiently large, remark $\ref{freq_then_dist}$ implies that 
\[d(C,\frac{1}{\sum_{i=1}^n\abs{T_i}}\sum_{i=1}^n\abs{T_i}\mu_{B_i})<\frac{\eps}{2},\]
and since $d(C,\mu_C)<\frac{\eps}{2}$, we also have
\[d(
\mu_C,\frac{1}{\sum_{i=1}^n\abs{T_i}}\sum_{i=1}^n\abs{T_i}\mu_{B_i})<\eps.\]
\end{proof}

We will use the following lemma proven in \cite{H}:
\begin{lem}\label{tilings}
Let $X$ be a zero-dimensional system with free action of an amenable group $G$. 
For any $\eps>0$, $\delta>0$ and $n\in\Nat$ there exists a disjoint, 
$(1-\eps)$-covering dynamical quasitiling $\CT$ such that every shape of $\CT$ 
is a $(1-\delta)$-subset of a F\o lner set $F_m$ with $m\geq n$.
\end{lem}

\begin{lem}\label{close_blocks}
Let $X$ be a zero-dimensional dynamical system (in array form) with the shift 
action of an amenable group $G$, and let $K$ be a face in the simplex 
$\CM_G(X)$. For any $\delta>0$ and $\eps>0$ there exists an $\eta$, $n$ and $j$ 
such that if $\CT$ is a disjoint, $(1-\eta)$-covering dynamical quasitiling by 
$(F_n,\eta)$-invariant sets, and $\CB$ denotes the family of all 
$(\CT,j)$-blocks $B$ such that $d(B,K)>\delta$, then $\sum_{B\in 
\CB}\mu(B)\abs{B} \leq \eps$ for every $\mu\in K$.
\end{lem}
\begin{proof}
Let $F$ denote the (closed) complement of the open $\delta$-ball around $K$ in 
$\CM(X)$ (note that we use here the space of all probability measures, not the space of invariant measures). Obviously, $\{\mu_B:B\in\CB\}\subset F$ for sufficiently large $n$ and $j$. For every $\alpha$, consider the set 
$V_\alpha\subset\CM(X)$ consisting of measures $\mu$ with the following 
property: if $\mu=\int_{\CM(X)}\nu d\xi$, and $\xi$ is supported by the 
closed $\alpha$-neighborhood of $\CM_G(X)$, then $\xi(F)<\eps$. Since $F$ is 
closed, 
the portmanteau lemma implies that $V_\alpha$ is an open set, and if $\alpha$ 
is small enough, $V_\alpha$ has to contain $K$ (otherwise, letting $\alpha$ 
tend to $0$, we could find a measure in $K$ that is a barycenter of a 
distribution on $\CM(X)$ not supported entirely by $K$, which is not 
possible). Let $\gamma$ be small enough 
that the open $\gamma$-neighborhood of $K$ is contained in $V_\alpha$. 
For fixed $\eps$ and $\delta$ choose $\eta$, $n$ and $j$ so that lemma 
\ref{concatenation} applied to a $(F_n,\eta)$-quasitiling yields the error of 
approximation $\gamma/2$ and let $\CT$ be such a quasitiling. Making $n$ and 
$j$ large 
enough, we can also assume that every block with domain $S\times 
\set{1,\ldots,j}$ (where $S$ is a shape of $\CT$) that occurs 
in $X$ lies in the $\alpha$-neighborhood of the set of invariant measures on 
$X$.
Note that the union $\bigcup\CB$ of the collection of all elements of $\CB$ (as defined in the statement of the lemma) is clopen, and thus the function $\mu\mapsto\mu(\bigcup\CB)$ is continuous on the set $\CM(X)$. 
Suppose that $\mu$ is an ergodic measure in $K$ such that $\sum_{B\in 
\CB}\mu(B)\abs{B}>\eps$. The function $\nu\mapsto \sum_{B\in \CB}\nu(B)\abs{B}$ 
is continuous, therefore if $\nu$ is close enough to $\mu$, then $\sum_{B\in 
\CB}\nu(B)\abs{B}>\eps$. In particular, by corollary \ref{approx_ergodic} we 
can find a block $C$ occurring in $X$, such that 
$d(\mu_C,\mu)<\frac{\gamma}{2}$, and $\sum_{B\in 
\CB}\mu_C(B)\abs{B}>\eps$. We can also assume that the union of tiles of $\CT$ 
contained in the domain of $C$ is a $(1-\eta)$-subset of $C$.
% The left-hand side of the latter inequality equals $\sum_{B\in\CB}\fr_C(B)\abs{B}$. 
By lemma \ref{concatenation}, $\mu_C$ is closer than $\frac{\gamma}{2}$ to 
$\nu=\frac{1}{\sum_{i=1}^n\abs{B_i}}\sum_{i=1}^n \abs{B_i}\mu_{B_i}$, where 
$B_1,...,B_n$ are all $\CT$-blocks occurring in $C$. For all $i$ such that 
$B_i\in\CB$, we 
have $\delta_{\mu_{B_i}}(F)=1$, so for 
$\xi=\frac{1}{\sum_{i=1}^n\abs{B_i}}\sum_{i=1}^n\abs{B_i}\delta_{\mu_{B_i}}$ we 
have $\xi(F)\geq\eps$. Since $\nu$ is in $V_\alpha$ (where 
such a decomposition should not exist), this is a contradiction.
\end{proof}

\section{Proof of the main result}

\begin{proof}[Proof of Theorem \ref{thm:main}]
Let $K$ be a face in $\CM_G(X)$. Recall that $X$ is represented as an array system, i.e. it is a subset of $Z=\prod_{j\in\Nat}{\Lambda_j}^G$, where $|\Lambda_j|<\infty$.
Fix a decreasing sequence $\eps_t$ such that 
$\sum_{t=1}^{\infty}\eps_t<\infty$. We will construct a sequence of maps 
$\phi_t$ on $X$, which are all going to be invertible continuous coding maps. 
Then we will prove that the sequence of maps $\Phi_t:\CM(X)\to \CM(Z)$, $\Phi_t(\mu)=\mu\circ\phi_t^{-1}$, converges uniformly on $K$ to an affine homeomorphism $\Phi$, while $\phi_t$ converge pointwise on a set of full measure to a map establishing an isomorphism between $(X,\mu)$ and some $(Y,\Phi(\mu))$ for each $\mu\in K$.
The construction will be inductive: let $\phi_0$ be the identity map. Now, 
supposing we have constructed a map $\phi_{t-1}$, let $X_{t-1}=\phi_{t-1}(X)$ 
(since $\phi_{t-1}$ is an invertible coding map, $X_{t-1}$ is conjugate to 
$X$).  There exist $k_t$, $n_t$ and some $\delta_t>0$ such that if $F$ is 
$(1-\delta_t)$-subset of a F\o lner set $F_m$, with $m\geq n_t$, and $B$ is a 
block with domain $F\times[1,k_t]$ occurring in $X_{t-1}$, then the distance 
between $\mu_B$ and $\CM_G(X_{t-1})$ is less than $\eps_t$. Let $\CT_t$ be a 
disjoint, $(1-\delta_t)$-covering, dynamical tiling of $X_{t-1}$, whose shapes are all $(1-\delta_t)$ 
subsets of F\o lner sets with indices greater than $n_t$ (see lemma \ref{tilings}). We can also assume 
that $\CT_{t}$ is congruent with $\CT_{t-1}$, and that $\CT_t$ together with 
$k_t$
satisfy the statement of lemma \ref{close_blocks}, i.e that if $\CB$ denotes 
the family of all $(\CT_t,k_t)$-blocks $B$ such that $d(B,\Phi_{t-1}(K))>\delta_t$, 
then $\sum_{B\in\CB}\mu(B)\abs{B}<\eps_t$ for every $\mu\in\Phi_{t-1}(K)$.

For every $(\CT_t,k_t)$-block $B$ the probability measure $\mu_B$ is closer than $\eps_t$ to some $\mu\in \CM_G(X_{t-1})$. By fact \ref{pet_consequence}, if $n_t$ is large enough, we can assume that for every shape $S$ of $\CT_t$ there exists a block $B_S$ with domain $S$, such that $\mu_{B_S}$ is closer than $\eps_t$ to some $\mu\in \Phi_{t-1}(K)$. We shall define a map $\tilde\phi_t$ as follows: for any $x\in X_{t-1}$ and any $T\in \CT_t(x)$, let $S$ be the shape of $T$ and let $B=x[T\times[1,k_t]]$. If the distance between $B$ and $\phi_{t-1}(K)$ is more than $\delta_t$, replace $x[T\times[1,k_t]]$ with $B_S$. Otherwise, $\tilde{\phi}_t$ introduces no changes. By doing this for all $T\in \CT_t(x)$, we obtain a new array, $\tilde{\phi}_t(x)$. 

Observe that if $x$ is in the support of any measure $\mu\in \Phi_{t-1}(K)$, 
then (by lemma \ref{close_blocks}) the union of tiles $T\in \CT_t(x)$ such that 
$x[T\times[1,k_t]]$ is a block distant by more than $\delta_t$ from $\Phi_{t-1}(K)$ has 
upper Banach density less than $\eps_t$, therefore $\tilde\phi_t(x)$ differs 
from $x$ on a set of coordinates of density less than $\eps_t$. This also means 
that the set of points $x\in X_{t-1}$ such that $\tilde\phi_t(x)$ differs from 
$x$ in column $e$ also has measure $\mu$ less than $\eps_t$ for $\mu\in 
\Phi_{t-1}(K)$.

Now let $\phi_t=\tilde{\phi}_t\circ\phi_{t-1}$. Since $\phi_t$ makes no changes in rows with indices $k_t$ and greater (and they allow us to determine the content of rows $0$ through $k_t$), it is a conjugacy. Furthermore, let $X_t=\phi_t(X)$ and let $\nu$ be an ergodic measure in $\CM_G(X_t)=\Phi_t(\CM_G(X))$. By Corollary \ref{approx_ergodic} for sufficiently large $n$ there is $x\in X_t$ such that $\mu_{x[C]}$, $C=F_n\times [1,k_t]$, is close to $\nu$. 
By the construction of $\phi_t$, every $(\CT_t,k_t)$-block in $x$ is closer than $\eps_t$ to some $\mu\in\Phi_{t-1}(K)$. If $F_n$ is a set sufficiently far in the \Fo sequence, then $x[C]$ is a block that is close to being a concatenation of $\CT_t$ blocks (the union of tiles of $\CT_t$ contained in $F_n$ is a $(1-\delta_t)$-subset of $F_n$). Therefore, by Lemma \ref{close_blocks} the measure $\mu_C$ differs by less than $\eps_t$ from $\frac{1}{\sum_{i=1}^n\abs{T_i}}\sum_{i=1}^n \abs{T_i}\mu_{x[B_i]}$, where $B_i=T_i\times [1,k_t]$.  Since each $x[B_i]$ is $\eps_t$-close to $\mu_{x[B_i]}$ the combination is $2\eps_t$-close to measure in $\Phi_{t-1}(K)$.

We will show that the maps $\Phi_t$ converge uniformly on $K$. To this end, it suffices to uniformly estimate the distance between $\Phi_t(\mu)$ and $\Phi_{t-1}(\mu)$ for ergodic $\mu\in K$ by a summable sequence. By lemma $\ref{close_blocks}$, for any $\mu\in K$ we have the estimate $\sum_{B\in \CB}\mu(B)\abs{B}<\eps_t$ for every $\mu\in K$, where $\CB$ denotes the family of all $\CT$-blocks $B$ such that $d(B,K)>\delta_t$. This implies that if $x\in X$, then the set of coordinates in $\phi_{t-1}(x)$ belonging to tiles of $\CT_t$ that are domains of blocks from $\CB$ has upper Banach density less than $\eps_t$. Since $\tilde{\phi}_t$ only makes any changes on these coordinates, $\phi_t(x)$ differs from $\phi_{t-1}(x)$ on a set of density less than $\eps_t$. If $x$ is in the support of some invariant measure $\mu$, then $\phi_{t-1}(x)$ and $\phi_t(x)$ are in the support of $\Phi_{t-1}(\mu)$ and $\Phi_t(\mu)$, respectively, and since the two points agree on a set of large lower Banach density, the measures are also close.

This uniform convergence, together with the fact that $\Phi_t(\CM_G(X))$ is within the $2\eps_t$-neighborhood of $\Phi_{t-1}(K)$, implies that $\Phi(\CM_G(X))\subset\Phi(K)$, and since the other inclusion is obvious, the two sets are equal. 

Now, define the set $Y$ (which will support the desired assignment) as follows:

\[Y=\bigcap_{s=1}^\infty\overline{\bigcup_{t=s}^\infty X_t}.\]
Observe that $Y$ is a closed, shift-invariant set, and that for any \Fo set $F$ and any $k\in\Nat$ every block with domain $F\times[1,k]$ in $Y$ occurs in infinitely many of the sets $X_t$. It follows that every invariant measure on $Y$ can be approximated by invariant measures on the $X_t$'s, and thus the set of invariant measures on $Y$ is contained in $\Phi(\CM_G(X))=\Phi(K)$. The other inclusion is generally true: for any weakly* convergent sequence of measures $\mu_t$ supported by $X_t$, the limit measure $\mu$ is always supported by $\bigcap_{s=1}^\infty\overline{\bigcup_{t=s}^\infty X_t}$. Therefore $\CM_G(Y)=\Phi(K)$.

Now, observe that for any $\mu\in K$ the set of points $x\in X_{t-1}$ such that the column $x(e)$ is modified by $\tilde{\phi}_t$ has measure $\mu$ less than $\eps_t$, because $\tilde\phi_t$ commutes with the shift map and for any $x$ in the support of $\mu$ the set of modified coordinates has upper Banach density less than $\eps_t$. Since the sequence $\eps_t$ is summable, the Borel-Cantelli lemma implies that for almost every $x\in X$ the columns $\phi_t(x)(e)$ are all equal from some point onwards. By shift-invariance, the same is true for $\phi_t(x)(g)$ for any $g$, so ultimately we conclude that if $\mu\in K$, then for $\mu$-almost every $x\in X$ every coordinate of $x$ is only changed finitely many times. This means that a limit point $\phi(x)$ is then well-defined, and this map $\phi$ is invertible (since every $\phi_t(x)$ retains the original contents of $x$ in the bottom row). In other words $\phi$ is an isomorphism between the measure-theoretic dynamical systems $(X,\mu)$ and $(Y,\Phi(\mu))$. 

\end{proof}
\section{Concluding remarks}
Firstly, we note that we can strengthen theorem \ref{thm:main} by combining it 
with theorem 1.2 of 
\cite{FH}, obtaining the following version:
\begin{thm}\label{thm:main2}
	Let $X$ be a Cantor system with free action of an amenable group $G$ and 
	let $K$ be a face in the simplex $\CM_G(X)$ of $G$-invariant measures of 
	$X$. There exists a Cantor system $Y$ with \emph{minimal} free action of 
	$G$, such that the natural assignment on $Y$ is equivalent to the identity 
	assignment on $K$.
\end{thm}
Secondly, note that the result is this paper is not strictly a strengthening of 
the main result in \cite{D2}, since while we gain the result for actions of 
amenable groups, we add the requirement that the action be free, whereas the 
original result merely requires that the face in question contain no periodic 
measures. Unfortunately, it is very much unclear how the machinery used to deal 
with periodic points would transfer to the group case, which is why the matter 
of directly 
extending the result in \cite{D2} remains open.

\end{document}